\documentclass{amsproc}
\usepackage{amsmath,amsthm,amssymb,amscd}
\usepackage[arrow,matrix]{xy}
\usepackage[dvips]{graphicx}
\usepackage{enumerate}
\usepackage{mathrsfs}
\usepackage{url} 

\theoremstyle{plain}
\numberwithin{equation}{section}
\newtheorem{thm}{Theorem}[section]

\newtheorem{prop}[thm]{Proposition}
\newtheorem{cor}[thm]{Corollary}
\newtheorem{lem}[thm]{Lemma}
\theoremstyle{definition}
\newtheorem{dfn}[thm]{Definition}
\newtheorem{exm}[thm]{Example}

\newtheorem{rem}[thm]{Remark}

\def\dim{\mathop{\mathrm{dim}}\nolimits}

\def\Ad{\mathop{\mathrm{Ad}}\nolimits}

\def\Cl{\mathop{\mathrm{Cl}}\nolimits}
\def\h{\mathcal{H}}
\def\CC{\mathbb{C}}
\def\BB{\mathbf{B}}
\def\QQ{\mathbb{Q}}
\def\RR{\mathbb{R}}
\def\ZZ{\mathbb{Z}}

\def\gr{\mathrm{Gr}}

\def\calB{\mathcal{B}}

\def\calM{\mathcal{M}}

\def\cl{\mathrm{cl}}

\def\frakg{\mathfrak{g}}

\def\nn{\mathbf{n}}
\def\hh{\mathbf{h}}

\def\Im{\mathop{\mathrm{Im}}\nolimits}
\def\bs{\backslash}

\def\Aut{\mathop{\mathrm{Aut}}\nolimits}

\def\Span{\mathrm{span}}
\def\Sym{\mathrm{Sym}}

\def\ev{\mathrm{ev}}

\def\od{\mathrm{od}}
\def\tr{\mathrm{tor}}
\def\tor{\mathrm{tor}}
\def\Sat{\mathrm{Sat}}
\def\scrH{\mathscr{H}}

\def\rt-m{\sqrt{m}}
\def\00{\mathbf{0}}
\def\Ker{\mathop{\mathrm{Ker}}\nolimits}

\title[Kato-Usui partial compactifications]{Kato-Usui partial compactifications\\ over the toroidal compactifications of Siegel spaces}
\author[T.~Hayama]{Tatsuki Hayama}
\address{Mathematical Sciences Center, Tsinghua University, Haidian District, Beijing 100084, China}
\email{tatsuki@math.tsinghua.edu.cn}
\dedicatory{dedicated to Professor Phillip A. Griffiths}
\date{\today}

\begin{document}
\maketitle
\begin{abstract}
We construct fans which give Kato-Usui partial compactifications of period domains of weight $-1$.
Similarly to the fans of toroidal compactifications, these fans are given by polyhedral decompositions.
We also demonstrate a fibration structure for some kinds of Kato-Usui boundary components.
\end{abstract}

\section{Introduction}
Let $\scrH$ be a Siegel upper half space, and let $\Gamma$ be an arithmetic subgroup.
By Mumford, et al. \cite{AMRT}, we have a toroidal compactification $\Gamma\bs \scrH_{\Sigma_{\tor}}$ 
by taking a suitable fan $\Sigma_{\tor}$.
Moreover, Carlson, Cattani and Kaplan \cite{CCK} showed a relationship between the boundary 
of $\Gamma\bs \scrH_{\Sigma_{\tor}}$ and degenerating Hodge structures.
These boundary points correspond to the nilpotent orbits, which generate limiting mixed Hodge 
structures by work of Schmid \cite{S}.

Generalizing the Siegel case, i.e. the case for Hodge structures of weight $1$, Kato and Usui \cite{KU} have introduced a construction to make partial compactifications of period domains.
Let $D$ be a period domain of (polarized pure) Hodge structure, as introduced by Griffiths \cite{G}.
Then the real Lie group $G_{\RR}$ acts on $D$ transitively.
A boundary point of the Kato-Usui partial comapactification is a nilpotent orbit generated by a nilpotent cone in $\frakg$.
Let $\Gamma$ be a subgroup of $G_{\ZZ}$, and let $\Sigma$ be a fan of nilpotent cones which is (strongly) 
compatible with $\Gamma$.
We can define the set $\Gamma\bs D_{\Sigma}$ of $\sigma$-nilpotent orbits for $\sigma\in \Sigma$ modulo $\Gamma$.
Here $\Gamma\bs D_{\Sigma}$ has a log geometrical structure and is a moduli space of log Hodge structures if $\Gamma$ is neat.

In the Siegel case, $\Sigma_{\tor}$ is compatible with an arithmetic subgroup $\Gamma$ and covers all possible nilpotent cones.
Beyond the Siegel case, however, we do not know how to construct such a big fans except for a few examples. In this paper, 
we construct a fan which covers all possible nilpotent cones of a special type.

\subsection{Even maps and odd maps}
We review \cite{H} to state our result in this paper.
Let $D$ be a period domain of weight $-1$, i.e. odd-weight by twisting.
Then $G_{\RR}$ is isomorphic to the symplectic group.
Here an isotropy subgroup $L$ of $G_{\RR}$ is compact.
Let $K$ be the maximal subgroup of $G_{\RR}$ containing $L$.
We then have the natural quotient map
$$
D\cong G_{\RR}/L\to G_{\RR}/K.
$$
Now $K$ is isomorphic to the unitary group.
Then $G_{\RR}/K$ can have a complex structure so that it is isomorphic to the 
Siegel space $\scrH$ or its complex conjugate $\bar{\scrH}$.
This map is compatible with the action of $G_{\ZZ}$.
Then we have the map 
\begin{align}\label{proj}
\Gamma\bs D\to \Gamma\bs \scrH,\quad \Gamma\bs D\to \overline{\Gamma\bs \scrH}
\end{align}
 for a subgroup $\Gamma$ of $G_{\ZZ}$.
 
For $F\in D$ we have the Hodge decomposition $H_{\CC}=\bigoplus _{p}H^{p,-p-1}$.
Regarding Hodge structures, the map $D\to \scrH$ (resp. $D\to \bar{\scrH}$) is given by
\begin{align}\label{int_D}
F\mapsto \bigoplus_{p:\text{even}} H^{p,-p-1}\quad (\text{resp. } \bigoplus_{p:\text{odd}} H^{p,-p-1}).
\end{align}
The abelian variety associated to the image of the above map is called the Weil intermediate Jacobian.
By taking Weil intermediate Jacobians, we have the map
\begin{align}\label{HStoA}
\{\text{Hodge structures of weight }-1\}\to \{\text{abelian varieties}\}.
\end{align}
Let $(W,F)$ be a $\ZZ$-mixed Hodge structure satisfying the following condition:
\begin{align}\label{MHScond}
&\gr^{W}_{p}=0 \text{ if } p\neq 0,-1,-2 \text{ and}\\
&\text{the Hodge structure on }\gr^{W}_{-1}\text{ is polarizable.}\nonumber
\end{align} 
Now we have the Deligne decomposition $H_{\CC}=\bigoplus_{p,q} I^{p,q}$.  
Then we may define a filtration $\tilde{F}$ such that
$$\tilde{F}^{-1}=H_{\CC},\quad \tilde{F}^0=(\bigoplus_{p:\text{ even}}I^{p,-p-1})\oplus (\bigoplus_p I^{p,-p}),\quad \tilde{F}^{1}=0.$$
Therefore $(W,\tilde{F})$ is a $\ZZ$-mixed Hodge structure of type (\ref{MHScond}) where 
$$h^{p,q}=0\quad \text{for }(p,q)\neq (0,0), (-1,0), (0,-1), (-1,-1).$$
By \cite{B}, $(W,\tilde{F})$ is a 1-motive.
Hence the map given by $(W,F)\mapsto (W,\tilde{F})$ induces
\begin{align}\label{MHStoMot}
\{\text{mixed Hodge structures of type (\ref{MHScond})}\}\to \{\text{1-motives}\},
\end{align}  
which is a generalization of the map (\ref{HStoA}).

Let $(\sigma,\exp{(\sigma_{\CC})}F)$ be a nilpotent orbit such that $N^2=0$ for $N\in\sigma$.
Then the limiting mixed Hodge structure$(W(\sigma),F)$ is a mixed Hodge structure of type (\ref{MHScond}).
By (\ref{MHStoMot}), we have a 1-motive $(W(\sigma),\tilde{F})$.
We say $\sigma$ is even-type (resp. odd-type) if $(\sigma,\exp{(\sigma_{\CC})}\tilde{F})$ (resp.  $(\sigma,\exp{(\sigma_{\CC})}\bar{\tilde{F}})$) is a nilpotent orbit for $\scrH$ (resp.$\bar{\scrH}$).
By Lemma \ref{even-odd}, this definition coincides with that defined in \cite{H} (Definition \ref{ev-od}).
For an even-type (resp. odd-type) cone $\sigma$, we have a map $D_{\sigma}\to \scrH_{\sigma}$ (resp.  $D_{\sigma}\to \bar{\scrH}_{\sigma}$) given by
\begin{align}\label{tilde}
(\sigma,\exp{(\sigma_{\CC})}F)\mapsto (\sigma,\exp{(\sigma_{\CC})}\tilde{F}) \quad (\text{resp. }(\sigma,\exp{(\sigma_{\CC})}\bar{\tilde{F}})),
\end{align}
where the map does not depend on the choice of $F$.
We denote the above map by $p^{\ev}$ (resp. $p^{\od}$).
The restriction $p^{\ev}|_D$ (resp. $p^{\od}|_D$) coincides with the one given by (\ref{int_D}). 
A $\sigma$-nilpotent orbit $(\sigma,\exp{(\sigma_{\CC})F})$ defines $\gr^{W(\sigma)}_{-1}F$ which does not depend on the choice of $F$ if $\sigma$ is even-type or odd-type.
The abelian variety associated to $\gr^{W(\sigma)}_{-1}\circ p^{\ev}(\sigma,\exp{(\sigma_{\CC})F})$ is the Weil intermediate Jacobian attached to $\gr^{W(\sigma)}_{-1}(\sigma,\exp{(\sigma_{\CC})F})$.

Since the monodromy representation acts by a real transformation, it commutes with the map $p^{\ev}$ (resp. $p^{\od}$).
For a fan $\Sigma$ of even-type (resp. odd-type) cones which is compatible with $\Gamma$, we then have the map
$$p^{\ev}:\Gamma\bs D_{\Sigma}\to \Gamma\bs \scrH_{\Sigma},\quad(\text{resp. }p^{\od}:\Gamma\bs D_{\Sigma}\to \overline{\Gamma\bs \scrH}_{\Sigma})$$
as an extension of the map (\ref{proj}).

While a toroidal partial compactification $\Gamma\bs \scrH_{\Sigma}$ is an analytic space, $\Gamma\bs D_{\Sigma}$ is not an analytic space but a log manifold.
A log manifold is a subspace of a log-analytic space defined by log differential forms.
The topology is the strong topology defined in \cite{KU}, which is stronger than the subspace topology of the analytic space.
We showed the continuity of $p^{\ev}$ and $p^{\od}$ in the $(1,1,1,1)$-case in \cite[\S \ref{exm}]{H}, however it is not known in more general settings.
We expect that these maps are continuous and morphisms of log structures.

We have not given a construction of fans in \cite{H}.
In this paper, I will construct of fans of even-type which are compatible with an arithmetic subgroup $\Gamma$.
By \cite{AMRT}, there exists a fan $\Sigma_{\tor}$ which makes $\Gamma\bs \scrH_{\Sigma_{\tor}}$ compact.
This fan is given by a family of polyhedral decompositions of the open cones of positive symmetric bilinear forms.
Now $\sigma\in\Sigma_{\tor}$ is even-type if  there exists a $\sigma$-nilpotent orbit, in which case 

\begin{align}\label{tor}
p^{\ev}:\Gamma\bs D_{\Sigma_{\tor}}\to\Gamma\bs \scrH_{\Sigma_{\tor}}.
\end{align}
However there might not exist a $\sigma$-nilpotent orbit (i.e., $D_{\sigma}=D$).
The map (\ref{tor}) is not necessarily surjective.
We will give a subfan $\Sigma_{\ev}\subset \Sigma_{\tor}$ of even-type cones and prove the map 
$$p^{\ev}:\Gamma\bs D_{\Sigma_{\ev}}\to \Gamma\bs \scrH_{\Sigma_{\ev}}\subset \Gamma\bs \scrH_{\Sigma_{\tor}}$$
is surjective.

In contrast to even-type, an odd-type fan is not obtained by the same construction of $\Sigma_{\ev}$ in general.
In fact, we will show that the odd-type fan does not contain a polyhedral decomposition in the $(1,1,1,1)$-case although $\Sigma_{\ev}$ is a set of polyhedral
decompositions.
See \S \ref{exm}.

\subsection{Relation to cycle spaces} 
We have the Satake compactification $\Gamma\bs \scrH_{\Sat}$ and the map 
$\zeta: \Gamma\bs \scrH_{\Sigma}\to \Gamma\bs \scrH_{\Sat}$.
We showed that $\zeta\circ p^{\ev}$ and $\bar{\zeta}\circ p^{\od}$ have a relation with cycle spaces in \cite{H}.

In this situation, where $G_{\RR}$ is a symplectic group,  the cycle space $\calM_D$ is isomorphic to $\scrH\times \bar{\scrH}$.
For an even-type (resp. odd-type) nilpotent orbit $(\sigma, \exp{(\sigma_{\CC})}F)\in D_{\Sigma}$, we have the $1$-variable SL(2)-orbit $(\rho,\phi)$ such that
$$\rho_*(\nn_-)=N,\quad \phi(0)=\hat{F},$$  
where $N$ is in the relative interior $\sigma^{\circ}$, $(W(\sigma),\hat{F})$ is the $\RR$-split MHS associated to the limiting mixed Hodge structure $(W(\sigma),F)$, and $(\nn_-,\hh,\nn_+)$ is the standard generator of $sl_2(\RR)$.
Now $\phi(i)\in D$ and we have the cycle $C_0\in\calM_D$ associated to $\phi(i)$.
We have 
$$X:=\frac{1}{2}\rho_*(i\nn_--\hh+i\nn_+),$$
which is in the $(-1,1)$-component of the Hodge decomposition of $\frakg_{\CC}$ with respect to $\phi(i)$.
We showed that $e^XC_0$ is in the closure $\calM_D^{\cl}$ of $\calM_D$.
Now $\scrH$ is isomorphic to the bounded domain $\calB$.
Then we have $\calM^{\cl}_D=\calB^{\cl}\times\bar{\calB}^{\cl}$ and the two projections
\begin{align*}
\pi^{\ev}:\calM_{D}^{\cl} \to \calB^{\cl},\quad \pi^{\od}:\calM_{D}^{\cl} \to \bar{\calB}^{\cl}.
\end{align*}
We showed $\pi^{\ev}(e^XC_0)$ (resp. $\pi^{\od}(e^XC_0)$) is in a Satake boundary component of $\calB$ and that 
$$\pi^{\ev}(e^XC_0)=\zeta\circ p^{\ev}(\sigma, \exp{(\sigma_{\CC})}F),\quad (\text{resp. }\pi^{\od}(e^XC_0)=\bar{\zeta}\circ p^{\od}(\sigma, \exp{(\sigma_{\CC})}F))$$  
in $\Gamma\bs \scrH_{\Sat}$.
Here $\pi^{\ev}(e^XC_0)$ (resp. $\pi^{\od}(e^XC_0)$) does not depend on the choice of $F$ and $N$.

\subsection{Boundary component structure}
As a generalization of period domains, Green, Griffiths and Kerr introducted Mumford-Tate domains \cite{GGK}.
In the same manner as the Kato-Usui partial compactifications, Kerr and Pearlstein \cite{KP} have constructed partial compactifications of Mumford-Tate domains.
Mumford-Tate domains are not algebraic in the non-classical situation \cite{GRT}, however the boundary components 
are sometimes ``arithmetic''.
In fact, Carayol \cite{C} studied the case where the Mumford-Tate group is $SU(2,1)$, which is also treated in the Griffiths Lectures \cite{G2}, and he showed that boundary components are isomorphic either to $\CC^*$, a CM elliptic curve, or its complex conjugate.
Here certain spaces of automorphic cohomology classes are isomorphic to the space of certain parabolic Picard modular forms by a Penrose-type transformation, and a Picard modular form has a Fourier expansion around a cusp whose coefficients are Theta functions defined over the CM elliptic curve.
These automorphic classes have generalized Fourier expansions similar to the Picard modular forms \cite[Lecture 10]{G2}.

To extend this result to general Mumford-Tate domains, Kerr and Pearlstein studied the geometric structure of boundary components of Mumford-Tate domains and showed that a boundary component has an ``arithmetic'' property if it satisfies some conditions.
In this paper, we will study a boundary component structure of period domains for even-type cones.

Let $\BB(\sigma)$ be the boundary component for $\sigma$.
If $\dim{\Im{N}}\leq h^{0,-1}$ for $N\in\sigma^{\circ}$, then $p^{\ev}|_{\BB(\sigma)}$ is surjective (Proposition \ref{surj}).
We will define a subspace $\BB'(\sigma)\subset \BB(\sigma)$ and show that the restriction  $p^{\ev}|_{\BB'(\sigma)}$ is a real-analytic fibration, therefore $\zeta\circ p^{\ev}|_{\BB'(\sigma)}$ is a fibration over the Satake boundary component.
In particular, $\BB'(\sigma)$ coincides with $\BB(\sigma)$ if $D$ is a period domain for a Hodge structure of level $3$.
Using this, we will give an estimate of the dimension of $\BB(\sigma)$.
Let $\Gamma_{\sigma}$ be the subgroup of $\Gamma$ stabilizing $\sigma$.
The map $\zeta\circ p^{\ev}$ gives the following theorem:  
\begin{thm}
$\Gamma_{\sigma}\bs \BB'(\sigma)$ is real-analytically fibered over a Siegel modular variety.
\end{thm}
A similar result is given by Kerr and Pearlstein \cite[Proposition 7.4]{KP}, where they show that certain boundary components are holomorphically fibered over Shimura varieties.
However the situation of the above theorem is different from their setting.
See Remark \ref{remKP}.

In addition, Kerr and Pearlstein showed a rigidity theorem for variations of Hodge structure.
By \cite[Proposition 10.6]{KP}, a limiting mixed Hodge structure determines a variation of Hodge structure.
Therefore limit points in boundary components have importance for variations of Hodge structure.
Moreover they showed finiteness of numbers of certain classes of variations of Hodge structure \cite[Theorem 10.4]{KP}.
It is expected that there is arithmetic significance at the limit points.
In fact, the limit point of the variation of Hodge structure for a quintic-mirror family at a maximal unipotent 
monodromy point is given by a value of a zeta function.

\subsection*{Acknowledgment}The author is grateful to the referee for his valuable advice and suggestions. 
This research is supported by the Mathematical Science Center, Tsinghua University.

\section{Construction of fans and their boundary component structures}
Let $H_{\ZZ}$ be a $\ZZ$-module of rank $2n$, and let $Q$ be a bilinear form on 
$H_{\ZZ}$ defined by
$$\begin{pmatrix}0&I\\-I&0\end{pmatrix}.$$ 
Let $\{h^{p,-p-1}\}_p$ be Hodge numbers for a Hodge structure of weight $-1$.
Then we have a period domain $D$ for $(H_{\ZZ},Q,\{h^{p,-p-1}\}_p)$.
Now 
$$G_{A}=\Aut{(H_{A},Q)}\cong Sp(n,A)$$
 for $A=\ZZ,\QQ,\RR,\CC$, and $G_{\RR}$ acts on $D$ transitively.
The isotropy subgroup $L$ and the maximal subgroup $K$ including $L$ are
$$L\cong \prod_p U(h^{p,-p-1}),\quad K\cong U(n).$$ 
We denote by $\scrH$ the Siegel space of degree $n$.
Then $\scrH\cong G_{\RR}/K$, and $\scrH$ is a period domain for 
$$(H_{\ZZ},Q,\{h^{0,-1}=h^{-1,0}=n\}).$$

\subsection{Polyhedral decompositions}
Let $S$ be an isotropic subspace of $H_{\RR}$ with respect to $Q$ defined over $\QQ$.
Now $0\leq\dim{S}\leq n$ and $S$ corresponds to the Satake boundary component of $\scrH$ (cf. \cite{N}).
For the Lie algebra $\frakg$ of $G_{\RR}$, we define the subalgebra
$$\eta(S)=\{N\in\frakg \; |\; \Im{N}\subset S\}.$$
For $X\in\frakg$, we have $Q(Xv,w)+Q(v,Xw)=0$ for $v,w\in H_{\RR}$ since $G_{\RR}$ preserves $Q$.
Then, for $N,N'\in \eta(S)$, 
$$Q(NN'v,w)=Q(N'v,Nw)=0$$
for all $v,w\in H_{\RR}$.
Therefore $N^2=0$ for $N\in \eta(S)$ and all elements in $\eta(S)$ commute with each other. 
We define
$$\eta^+(S)=\{N\in \eta\; |\; \phi_N>0 \text{ on }H_{\RR}/S^{\perp}\},$$
where $\phi_N$ is the bilinear form defined by $Q(\bullet , N\bullet)$, 
which defines a polarization on $H_{\RR}/S^{\perp}$ if $N$ generates a nilpotent orbit.

In this paper, every cone is assumed to be a finitely generated convex rational polyhedral cone.

\begin{dfn}[{\cite[Definition 7.3]{N}}]
Let $\Gamma$ be an arithmetic subgroup of $G_{\RR}$, and let $\Gamma(S)$ be the maximal subgroup of $\Gamma $ which stabilizes $S$.
A fan $\Sigma(S)$ of cones in the closure $\Cl{(\eta^{+}(S))}$ is said to be a $\Gamma(S)$-admissible polyhedral decomposition of $\eta^+(S)$ if:
\begin{itemize}
\item If $\sigma\in\Sigma(S)$ and $\gamma\in\Gamma(S)$, then $\gamma\sigma\in \Sigma(S)$;
\item There exists finitely many cones $\sigma_1,\ldots ,\sigma_{\ell}\in\Sigma(S)$ such that for any $\sigma\in\Sigma(S)$, there exists $\gamma\in \Gamma(S)$ with $\gamma\sigma=\sigma_j$  for some $j$;
\item $\eta^{+}(S)=\bigcup_{\sigma\in\Sigma(S)}(\sigma\cap \eta^{+}(S))$.
\end{itemize} 

Moreover a union $\Sigma=\bigcup_{Q(S,S)=0}\Sigma(S)$ of $\Gamma(S)$-admissible polyhedral decompositions is said to be a $\Gamma$-admissible family if:
\begin{itemize}
\item $\gamma\Sigma(S)=\Sigma(\gamma S)$;
\item If $S'\subset S$, $\Sigma(S')$ is the restriction of $\Sigma(S)$ on $\Cl(\eta^+(S'))$. 
\end{itemize} 
\end{dfn}

The existence of a $\Gamma$-admissible family of polyhedral decompositions is proved by \cite{AMRT}.

\begin{exm}\label{decomp}
Let $e_1,\cdots ,e_{2n}$ be a symplectic basis of $H_{\ZZ}$ such that $Q(e_{n+j},e_{j})=1$ for $1\leq j\leq n$.
We define an isotropy subspace
$$S=\Span_{\RR}\{e_{2n-1},e_{2n}\}.$$
Then 
\begin{align*}
\eta(S)=\left\{\begin{array}{l|l}\begin{pmatrix}\LARGE{0}&\footnotesize{\begin{matrix}0&0\cr 0&X\cr \end{matrix}}\cr \LARGE{0}&\LARGE{0}\cr\end{pmatrix}& X\in\Sym(2,\RR)\end{array}\right\}\cong \Sym(2,\RR).
\end{align*}
Therefore a $\Gamma(S)$-admissible polyhedral decomposition of $\eta^+(S)$ is given by 
\begin{align*}&\Sigma(S)=\Gamma(S)\{\text{faces of }\sigma_0\},\\
&\text{where }\sigma_0=\RR_{\geq 0}\begin{pmatrix}1&0\cr 0&0\end{pmatrix}+\RR_{\geq 0}\begin{pmatrix}0&0\cr 0&1\end{pmatrix}+\RR_{\geq 0}\begin{pmatrix}1&-1\cr -1&1\end{pmatrix}
\end{align*}
(cf. \cite[Theorem 8.7]{N}).
Here $\Gamma\Sigma(S)$ is a $\Gamma$-admissible family.
\end{exm}

\subsection{Even-type and odd-type fans}
Let $\sigma$ be a nilpotent cone in $\frakg$.
Then $\sigma$ can be written as
$$\sigma=\sum_{1\leq j \leq \ell}\RR_{\geq 0}N_j$$
with nilpotents $N_1,\ldots ,N_{\ell}\in\frakg_{\QQ}$.
For $F$ in the compact dual $\check{D}$, we call $(\sigma,\exp{(\sigma_{\CC})F})$ a nilpotent orbit if it satisfies the following conditions:
\begin{itemize}
\item $\exp{(\sum_j ix_jN_j)}F\in D$ if $x_1,\ldots ,x_{\ell}\gg 0$;
\item $NF^p\subset F^{p-1}$ for any $p$ and $N\in \sigma$.
\end{itemize}
Let $(\sigma,\exp{(\sigma_{\CC})F})$ be a nilpotent orbit.
By \cite{CK}, each element $N$ of the interior $\sigma^{\circ}$ of $\sigma$ determines the same monodromy weight fltration $W=W(\sigma)$, which we center at weight $-1$ so that $(W,F)$ is a mixed Hodge structure.
Now we have the Deligne decomposition $H_{\CC}=\bigoplus_{p,q}I^{p,q}$ for $(W,F)$ where
$$W_k=\sum _{p+q\leq k}I^{p,q},\quad F^{\ell}=\sum_{p\geq \ell}I^{p,q}.$$ 

\begin{dfn}[\cite{H}]\label{ev-od}
A nilpotent orbit $(\sigma,\exp{(\sigma_{\CC})F})$ is called even-type (resp. odd-type) if it satisfies the following conditions:
\begin{itemize}
\item $N^2=0$ for $N\in \sigma$;
\item $I^{p,-p}=0$ for any odd (resp. even) integer $p$ with respect to the LMHS $(W,F)$.
\end{itemize}
A nilpotent cone $\sigma$ is called even-type (resp. odd-type) if $\sigma$ generates a nilpotent orbit and all $\sigma$-nilpotent orbits are even-type (resp. odd-type).
A fan $\Sigma$ is called even-type (resp. odd-type) if any cone of $\Sigma$ is even-type (resp. odd-type).
\end{dfn}

By the self-duality of monodromy weight filtration (\cite[Lemma 6.4]{S}),
$$W_{-2}^{\perp}=W_{-1}$$
and hence $W_{-2}$ is an isotropy subspace.
Moreover, we have the following lemma:

\begin{lem}\label{even-odd}
For the relative interior $\sigma^{\circ}$ of $\sigma$, $\sigma^{\circ}\subset \eta^{+}(W_{-2})$ (resp. $\sigma^{\circ}\subset -\eta^{+}(W_{-2})$) if and only if $\sigma$ is even-type (resp. odd-type).
\end{lem}

\begin{proof}
If $\sigma$ is even-type (resp. odd-type), $N^2=0$ and then the weight filtration is given by
$$W_{0}=H_{\RR},\quad W_{-1}=\Ker{N},\quad W_{-2}=\Im{N},\quad W_{-3}=0$$
where $N\in \sigma^{\circ}$.
Therefore
$$H_{\CC}/(W_{-2,\CC})^{\perp}=H_{\CC}/W_{-1,\CC}\cong \bigoplus_{p:\text{even}} I^{p,-p}\quad (\text{resp. }\bigoplus_{p:\text{odd}} I^{p,-p}).$$
By the polarization condition of the LMHS, 
\begin{align}\label{pol}
i^{2p}Q(v,N\bar{v})=\begin{cases}Q(v,N\bar{v})>0&\text{if }p \text{ is even,}\\-Q(v,N\bar{v})>0&\text{if }p \text{ is odd,}\end{cases}
\end{align}
for $0\neq v\in I^{p,-p}$.
Then $\sigma^{\circ}\subset \eta^+(W_{-2})$ (resp. $\sigma^{\circ}\subset -\eta^+(W_{-2})$).

Moreover by the polarization condition, $Q(v,Nv)>0$ for $0\neq v\in I^{0,0}$ and by (\ref{pol}),
\begin{align}
Q(v+\bar{v},N(v+\bar{v}))=\begin{cases}2Q(v,N\bar{v})>0&\text{if }p \text{ is even,}\\2Q(v,N\bar{v})<0&\text{if }p \text{ is odd,}\end{cases}
\end{align}
for $0\neq v\in I^{p,-p}$ with $p\neq 0$.
Then $I^{p,-p}$ has to be $0$ for odd (resp. even) $p$ if $\sigma^{\circ}\subset \eta^{+}(W_{-2})$ (resp. $\sigma^{\circ}\subset -\eta^+(W_{-2})$). 
Therefore the lemma holds.
\end{proof}

\subsection{Maps to the toroidal compactifications}
We denote by $\BB(\sigma)$ (resp. $\BB_{\tr}(\sigma)$, $\bar{\BB}_{\tor}(\sigma)$)  the set of $\sigma$-nilpotent orbits for $D$ (resp. $\scrH$, $\bar{\scrH}$).
For an even-type (resp. odd-type) nilpotent cone $\sigma$, we have the map $p^{\ev}:\BB(\sigma)\to \BB_{\tr}(\sigma)$ (resp. $p^{\od}:\BB(\sigma)\to \bar{\BB}_{\tr}(\sigma)$) given by
(\ref{tilde}).

\begin{prop}\label{surj}
Let $\sigma$ be an even-type cone with $\dim{\Im{N}}\leq h^{0,-1}$ for all $N\in \sigma^{\circ}$.
Then $p^{\ev}:\BB(\sigma)\to\BB_{\tor}(\sigma)$ is surjective.
\end{prop}
\begin{proof}
For a nilpotent orbit $(\sigma,\exp{(\sigma_{\CC})}F)\in \BB_{\tor}(\sigma)$, it is enough to show that there exists a nilpotent orbit $(\sigma,\exp{(\sigma_{\CC})}\breve{F})\in \BB(\sigma)$ such that the image through $p^{\ev}$ is $(\sigma,\exp{(\sigma_{\CC})}F)$.
Let $m=\dim{S}$.
Then the Deligne decomposition for $(W(\sigma),F)$ is given by
$$H_{\CC}=I^{0,0}\oplus I^{-1,0}\oplus I^{0,-1}\oplus I^{-1,-1}, \quad \dim{I^{0,0}}=m,\quad \dim{I^{-1,0}}=n-m.$$
Now $\overline{I^{-1,0}}=I^{0,-1}$, then $I^{-1,0}\oplus I^{0,-1}$ is defined over $\RR$.
Let $H'$ be the subspace of $H_{\RR}$ such that $H'_{\CC}=I^{-1,0}\oplus I^{0,-1}$.
Then $I^{-1,0}$ determines the point in the period domain for 
$$(H',Q|_{H'},\{h'^{-1,0}=h'^{0,-1}=n-m,\; 0\text{ otherwise}\})$$
which is isomorphic to the Siegel space $\scrH'$ of degree $n-m$.
Let $D'$ be the period domain for $(H',Q|_{H'},\{h'^{p,-1-p}\})$ where
$$h'^{p,-1-p}=\begin{cases}h^{p,-1-p}&\text{if }p\neq -1,0,\\h^{p,-1-p}-m&\text{if }p=-1,0.\end{cases}$$
We then have 
\begin{align}\label{prime}
\scrH'\cong Sp(n-m,\RR)/U(n-m),\quad D'\cong Sp(n-m,\RR)/\prod_p U(h'^{p,-p-1}),
\end{align}
and we have the even map $p'^{\ev}:D'\to \scrH'$.
Now $I^{0,-1}\in \scrH'$ and there exists $F'\in D'$ such that $p'^{\ev}(F')=I^{0,-1}$.
We define $\breve{F}\in \check{D}$ by 
\begin{align}\label{breve}
\breve{F}^p=\begin{cases}F'^{p}&\text{if }p>0\\F'^{p}\oplus I^{0,0}&\text{if }p=0\\F'^{p}\oplus I^{0,0}\oplus I^{-1,-1}&\text{if }p<0.\end{cases}
\end{align}\label{brave}
Then $(W(\sigma),\breve{F})$ is a LMHS such that
$$\gr^{W(\sigma)}_{0,\CC}\cong I^{0,0},\quad \gr^{W(\sigma)}_{-1}(\breve{F})\cong F',\quad \gr^{W(\sigma)}_{-2,\CC}\cong I^{-1,-1}.$$
Therefore $(\sigma,\exp{(\sigma_{\CC})}\breve{F})$ is a $\sigma$-nilpotent orbit for $D$ and the image through $p^{\ev}$ is $(\sigma,\exp{(\sigma_{\CC})}F)$. 
\end{proof}
\begin{cor}
If $h^{0,-1}\geq 2$, $\Sigma(S)$ and $\Gamma\Sigma(S)$ of Example \ref{decomp} is an even-type fan. 
\end{cor}
\begin{rem}
The above proposition does not hold for an odd-type cone. See the example of \S\ref{exm}.
\end{rem}

For a fan $\Sigma$ of nilpotent cones, we define
$$D_{\Sigma}=\bigsqcup_{\sigma\in\Sigma}\BB(\sigma).$$
Then for an even-type (resp. odd-type) fan $\Sigma$ we can define the map
$$p^{\ev}:D_{\Sigma}\to \scrH_{\Sigma}\quad (\text{resp. }p^{\od}:D_{\Sigma}\to \bar{\scrH}_{\Sigma}),$$
and for a subgroup $\Gamma$ of $G_{\ZZ}$ which is compatible with $\Sigma$ we have
$$p^{\ev}:\Gamma\bs D_{\Sigma}\to \Gamma\bs \scrH_{\Sigma}\quad (\text{resp. }p^{\od}:\Gamma\bs D_{\Sigma}\to \overline{\Gamma\bs \scrH}_{\Sigma}),$$
where $p^{\ev}=id$ if $D=\scrH$.




%



Let $\Gamma$ be an arithmetic subgroup of $G_{\RR}$ and let $\Sigma_{\tor}$ be a $\Gamma$-admissible family of polyhedral decompositions.
Then $\Gamma\bs D_{\Sigma_{\tor}}$ is called a toroidal compactification.
A toroidal compactification is compact, and moreover it is smooth if $\Gamma$ is neat.
By Lemma \ref{even-odd}, for $\sigma\in \Sigma_{\tor}$, $\sigma$ is even-type and $-\sigma$ is odd-type if it has nilpotent orbit.
We then have the maps
\begin{align}\label{map}
p^{\ev}:\Gamma\bs D_{\Sigma_{\tor}}\to \Gamma\bs \scrH_{\Sigma_{\tor}}, \quad p^{\od}:\Gamma\bs D_{\Sigma^{-}_{\tor}}\to \overline{\Gamma\bs \scrH}_{\Sigma^{-}_{\tor}}
\end{align}
where $\Sigma^{-}_{\tor}$ is the fan of $-\sigma$ for all $\sigma\in \Sigma_{\tor}$.
As we will see at \S \ref{exm}, these maps may not be surjective.
   
Now 
$$\Sigma_{\tor}=\bigcup_{Q(S,S)=0}\Sigma(S).$$
We then define the subfan
\begin{align}\label{even}
\Sigma_{\ev}=\bigcup_{\begin{subarray}{c}Q(S,S)=0,\\ \dim{S}\leq h^{0,-1}\end{subarray}}\Sigma(S)
\end{align}
of $\Sigma_{\tor}$.
By Proposition \ref{surj}, we have the following corollary
\begin{cor}
$p^{\ev}:\Gamma\bs D_{\Sigma_{\ev}}\to \Gamma\bs\scrH_{\Sigma_{\ev}}$ is surjective.
\end{cor}


\subsection{Boundary component structure}
Boundary component structure of toroidal compactifications are well-known.
Using it, we show some boundary component structure for even-type cones. 
  
Let $\sigma$ be a cone such that $\sigma^{\circ}\subset\eta^+(S)$ for some isotropy subspace $S$.  
Now $\eta(S)\cong \Sym(m,\RR)$ with $m=\dim{S}$.
There exists a subspace $Z_{\sigma}$ of $\eta(S)_{\CC}$ satisfying $\eta(S)_{\CC}=Z_{\sigma}\oplus \sigma_{\CC}$.
Let $\BB_{\Sat}(S)$ be the Satake boundary component corresponding to $S$.
Then $\BB_{\Sat}(S)$ is isomorphic to the Siegel space $\scrH'$ of degree $n-m$.
We have
$$\exp{(\eta(S)_{\CC})}\scrH\cong \BB_{\Sat}(S)\times \CC^k\times Z_{\sigma}\times \sigma_{\CC}$$
where $k=m\times (n-m)$.
Here $\BB_{\tor}(\sigma)\cong \exp{(\sigma_{\CC})}\scrH/\exp{(\sigma_{\CC})}$ which is an open subspace of 
\begin{align}\label{siegel3}
\exp{(\eta(S)_{\CC})}\scrH/\exp{(\sigma_{\CC})}\cong \BB_{\Sat}(S)\times \CC^k\times Z_{\sigma}
\end{align}
(cf. \cite[\S 5]{H1}).

For $(\sigma,\exp{(\sigma_{\CC})}F)\in \BB_{\tor}(\sigma)$, the projection $\zeta :\BB_{\tor}(\sigma)\to\BB_{\Sat}(\sigma)\cong\scrH'$ is given by
$$(\sigma,\exp{(\sigma_{\CC})}F)\to\gr^{W(\sigma)}_{-1}(F).$$ 
Now for $D'$ of (\ref{prime}) and the even map $p'^{\ev}:D'\to\scrH'\cong \BB_{\Sat}(S)$, 
we have the fibered product
$$\xymatrix{
\BB'(\sigma)\ar@{->}[d]\ar@{->}[r]&\BB_{\tor}(\sigma) \ar@{->}[d]^{\zeta}\\
D'\ar@{->}[r]_{p'^{\ev}}&\BB_{\Sat}(\sigma)
}.
$$
If $m\leq h^{0,-1}$, by the construction of $\breve{F}$ of (\ref{brave}), we have
$$\BB'(\sigma)\hookrightarrow \BB(\sigma);\quad (F',(\sigma,\exp{(\sigma_{\CC})}F))\mapsto (\sigma,\exp{(\sigma_{\CC})}\breve{F}).$$
As a subspace of $\BB(\sigma)$, we have
$$\BB'(\sigma)=\{(\sigma,\exp{(\sigma_{\CC})}F)\in\BB(\sigma)\;|\; h^{p,-p}=0\text{ for }(W(\sigma),F)\text{ if }p\neq 0\}.$$
In fact, the maps
$$p^{\ev}|_{\BB'(\sigma)}:\BB'(\sigma)\to\BB_{\tor}(\sigma),\quad \gr^{W(\sigma)}_{-1}:\BB'(\sigma)\to D'$$
are those appearing in the fibered product diagram above. 
If the level of Hodge structure is $3$, i.e. $h^{p,-p-1}=0$ if $p>1$, a possible even-type nilpotent orbit is of type of $\BB'(\sigma)$.
Therefore $\BB(\sigma)=\BB'(\sigma)$ if $D$ is a period domain for a Hodge structure of level $3$. 

Using this boundary component structure, we have an estimate of the dimension of $\BB(\sigma)$.
The fiber of $D'\to \scrH'$ is the compact complex manifold
$$C'=U(n-m)/\prod_{p\geq 0} U(h'^{p,-p-1})$$
where $h'^{p,-p-1}$ is the Hodge number for $D'$ and
$$\dim{\BB(\sigma)}\geq\dim{\BB'(\sigma)}=\dim{\BB_{\tor}(\sigma)}+\dim_{\CC}{C'}.$$
Here $\dim{\BB_{\tor}(\sigma)}$ can be calculated by (\ref{siegel3}), which yields 

\begin{prop}  
Let $\sigma $ be an even-type nilpotent cone with $m=\dim{\Im{N}}\leq h^{0,-1}$ for $N\in\sigma^{\circ}$. Then 
\begin{align*}
\dim{\BB(\sigma)}\geq \frac{1}{2}(n-m)(n-m+1)+m(n-m)+\frac{1}{2}m(m+1)-\dim_{\RR}{\sigma_{\RR}}\\
+\frac{1}{2}\large((n-m)^2-\sum_{p\geq 0}(h'^{p,-p-1})^2\large).
\end{align*}
If $D$ is a period domain for a Hodge structure of level $3$, the above inequality is an equality. 
\end{prop}

Let $\Gamma$ be a neat arithmetic subgroup.
We define
$$P_S=\{g\in G_{\RR}\;|\; gS=S\},\quad P_{\sigma}=\{g\in G_{\RR}\;|\; \Ad{(g)}\sigma=\sigma\}.$$
Then $P_{\sigma}$ is a subgroup of $P_S$.
By \cite[Proposition 4.10]{N}, we have surjective maps
$$p_{h}:P_{S}\to \Aut{(\BB_{\Sat}(S))},\quad p_{\ell}: P_{S}\to \Aut{(\eta^+(S),\eta(S))}$$
and $P_S$ can be written as the semi-direct product
\begin{align}\label{semi-direct}
P_S\cong W_S\rtimes(p_h(P_S)\times p_{\ell}(P_S))
\end{align}
where $W_{S}$ is the unipotent radical of $P_S$.
The maps $p_{h}$ and $p_{\ell}$ satisfy  
$$\zeta(gx)=p_{h}(g)\zeta (x),\quad \Ad{(g)}y=p_{\ell}(g)y$$
for $g\in P_S$, $x\in \BB_{\tor}(\sigma)$ and $y\in\eta(S)$.
Then $\Gamma_{\sigma}=\Gamma\cap P_{\sigma}$ can be written as a semi-direct product
$$\Gamma_{\sigma}\cong (W_S\cap\Gamma(S))\rtimes \left( p_h(\Gamma(S))\times (p_{\ell}(\Gamma_{\sigma}))\right).$$
Now $\Gamma'=p_{h}(\Gamma (S))$ is an arithmetic subgroup of $\Aut{(\BB_{\Sat}(S))}\cong Sp(n-m,\RR)$ and then $\Gamma'\bs \BB_{\Sat}(\sigma)$ is a Siegel modular variety.
Forming the quotient, we obtain the fibered product
$$\xymatrix{
\Gamma_{\sigma}\bs\BB'(\sigma)\ar@{->}[d]\ar@{->}[r]&\Gamma_{\sigma}\bs\BB_{\tor}(\sigma) \ar@{->}[d]\\
\Gamma'\bs D'\ar@{->}[r]&\Gamma'\bs\BB_{\Sat}(\sigma)
}
$$

\begin{lem} 
$p'^{\ev}:\Gamma'\bs D'\to \Gamma'\bs \BB_{\Sat}(\sigma)$ is a fiber bundle whose fiber is $C'$.
\end{lem}

\begin{proof}
For $F'_0\in D'$, we have the maximal compact subgroup $K'_0$ containing the isotropy subgroup at $F'_0$.
The orbit $K'_0F'_0$ is the fiber of $p'^{\ev}:D'\to\BB_{\Sat}(\sigma)$ at $p'^{\ev}(F'_0)$.
If there exists $1\neq g\in K'_0\cap\Gamma'$, $\{g,g^2,g^3,\ldots\}$ gives an infinite sequence in $K'_0$  since $\Gamma'$ is neat.
Now since $K'_0$ is compact, there exists a subsequence which converges to a point in $K'_0$.
However this contradicts the proper discontinuity of the action of $\Gamma'$.
\end{proof}
In conclusion, we have the following theorem:
\begin{thm}\label{main}
$p^{\ev}:\Gamma_{\sigma}\bs \BB'(\sigma)\to\Gamma_{\sigma}\bs\BB_{\tor}(\sigma)$ is a real-analytic fibration.
Moreover $\Gamma_{\sigma}\bs \BB'(\sigma)$ is fibered over a Siegel modular variety.
\end{thm} 

\begin{rem}\label{remKP}
In \cite[Proposition 4.2]{KP}, Kerr and Pearlstein showed that a Mumford-Tate group for a boundary component can be written as a semi-direct product similarly to (\ref{semi-direct}).
Moreover they showed a fibration structure for some boundary components in \cite[Proposition 7.4]{KP}.
If $M(\RR)=G_{\RR}$, the Mumford-Tate group $M_{\BB(\sigma)}(\RR)$ for the boundary component $\BB(\sigma)$ is contained in $P_{\sigma}$, and the Lie subgroup $G_{\BB(\sigma)}(\RR)$ of \cite[Definition 4.1]{KP} contains $p_h(P_S)$, which acts on 
$D'$ transitively.
Then the Lie algebra $\frakg_{\BB(\sigma)}$ is not of the type $(-1,1)+(0,0)+(1,-1)$ because of the Hodge numbers of $D'$ unless $D=\scrH$.
Therefore, our situation is outside the setting of \cite[Proposition 7.4]{KP} if $D\neq \scrH$.
If $D= \scrH$, then $\BB(\sigma)=\BB_{\tor}(\sigma)$. 
\end{rem}

\subsection{Example}\label{exm}
We consider the case where $h^{p,-p-1}=1$ if $p=1,0,-1,-2$, $h^{p,-p-1}=0$ otherwise (the case for Hodge structures of Calabi-Yau threefolds with $h^{2,1}=1$).
In this case
$$G_{\RR}\cong Sp(2,\RR),\quad L\cong U(1)\times U(1),\quad K\cong U(2).$$
In  \cite[\S 12.3]{KU}, this case is well-studied.
Any nilpotent cone in this case is rank $1$, and its generator $N$ is classified as follows:
\begin{itemize}
\item[(I)] $N^2=0,\;\dim{(\Im{N})}=1$;
\item[(II)] $N^2=0,\;\dim{(\Im{N})}=2$;
\item[(III)] $N^3\neq 0,N^4=0$.
\end{itemize}
LMHS of type-I and type-II are described as follows:
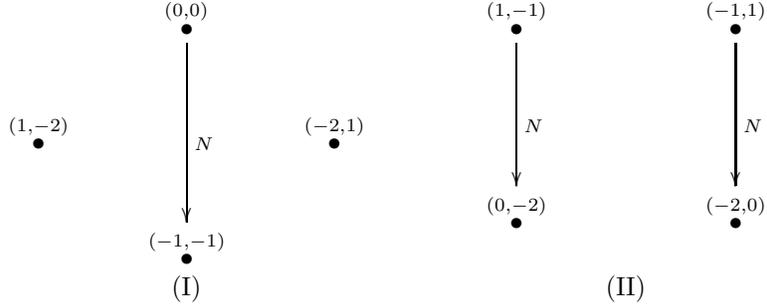
\begin{center}
\begin{tabular}{ccc}
$
\xymatrix{
&\stackrel{(0,0)}{\bullet}\ar@{->}[dd]^N& \\ \stackrel{(1,-2)}{\bullet}& & \stackrel{(-2,1)}{\bullet} \\ &\stackrel{(-1,-1)}{\bullet}&
}$&$\qquad$&
$\xymatrix{
\stackrel{(1,-1)}{\bullet}\ar@{->}[dd]^N&& \stackrel{(-1,1)}{\bullet}\ar@{->}[dd]^N\\ \\\stackrel{(0,-2)}{\bullet}&& \stackrel{(-2,0)}{\bullet} 
}$\\
(I)&$\qquad$
&(II)\\
\end{tabular}
\end{center}
Then a type-I cone is even-type and a type-II cone is odd-type.
The fan of all possible type-I cones is
$$\Sigma_1=\{\RR_{\geq 0}N\;|\; N\in \frakg_{\QQ},\; N^2=0,\; \dim{(\Im{N})}=1,\; \phi_N>0 \text{ on }H_{\RR}/\Ker{N} \}.$$
For a generator $N$ of $\sigma\in\Sigma_1$, $\eta(\Im{N})$ is $1$-dimensional, and then $\eta^+(\Im{N})=\RR_{> 0}N$.
Then $\Sigma_1=\Sigma_{\ev}$ of (\ref{even}).

On the other hand, for the fan $\Sigma_2$ of all possible type-II cones, all cones of $\Sigma_2$ are rank $1$.
For a generator $N$ of $\sigma\in\Sigma_2$, $\dim{(\Im{N})}=2$ and $\dim{(\eta(\Im{N}))}=3$.
Then $\Sigma_2$ does not contain a polyhedral decomposition of $\eta^+(\Im{N})$.

For a type-I cone $\sigma_1$, $\BB(\sigma_1)$ coincides with $\BB'(\sigma_1)$.
Now if $D'$ is the period domain with $h^{1,-2}=h^{-2,1}=1$, then $D'$ is isomorphic to the complex 
conjugate $\bar{\scrH}_1$ of the upper half plane.
Moreover $\BB_{\tor}(\sigma_1)\cong \scrH_1\times \CC$ and $\BB_{\Sat}(\sigma_1)\cong \scrH_1$.
Therefore the boundary component $\BB(\sigma_1)$ is the fiber product
$$\xymatrix{
\BB(\sigma_1)\ar@{->}[d]\ar@{->}[r]&\scrH_1 \times \CC \ar@{->}[d]^{\mathrm{proj.}}\\
\bar{\scrH}_1\ar@{->}[r]_{\mathrm{conj.}}&\scrH_1
}
$$
If we take a quotient by a neat subgroup $\Gamma_{\sigma_1}$, the right vertical arrow 
\begin{align}\label{modular}
\Gamma_{\sigma_1}\bs (\scrH_1\times \CC)\to \Gamma'\bs \scrH_1
\end{align}
is the canonical elliptic fibration over the Siegel modular curve $\Gamma'\bs \scrH_1$.

In this case, $\Aut{D'}\cong SL(2,\RR)$ is $\gr^{W}_0M_{B(N_2)}(\RR)$ of \cite[Example 8.2]{KP}.
They define a complex structure on $SL(2,\RR)/U(1)$ as $SL(2,\RR)/U(1)\cong\scrH_1$.
Therefore, in \cite[Example 8.2]{KP}, $\BB(\sigma_1)\cong \scrH_1\times \CC$ and $\Gamma_{\sigma_1}\bs \BB(\sigma_1)$ is holomorphically fibered over the Siegel modular curve through (\ref{modular}), which has different complex structure from the one we defined.
In our setting, $p'^{\ev}:D'\to \BB_{\Sat}(\sigma_1)$ is complex conjugation, and then $\BB(\sigma_1)$ is not holomorphically but real-analytically fibered. 
The restriction $p^{\ev}$ to each even-type boundary component gives a real-analytic fiber bundle, 
however $p^{\ev}$ itself is not a fiber bundle.  

For a type-II cone $\sigma_2$, we have $\BB(\sigma_2)\cong \CC\times \{\pm 1\}$ by \cite[\S 12.3]{KU}.
Then
$$\dim{\BB(\sigma_2)}<\dim{\bar{\BB}_{\tor}(\sigma_2)}=2.$$
Therefore the map $p^{\od}$ on the boundary component is not surjective.
By \cite[Example 8.2]{KP}, $\Gamma_{\sigma}\bs\BB(\sigma_2)$ is a CM elliptic curve.






\begin{thebibliography}{ABCD}
\bibitem[AMRT]{AMRT} A. Ash, D. Mumford, M. Rapoport and Y. S. Tai, \it{Smooth compactification of locally symmetric varieties}, \em{Math.\ Sci.\ Press, Brookline}, 1975.
\bibitem[B]{B} L. Barbieri-Viale, {\it On the theory of 1-motives} in Algebraic Cycles and Motives London Mathematical Society Lecture Note Series, {\bf 343}, Cambridge University Press, London, 2007. 
\bibitem[C]{C} H. Carayol, {\it Cohomologie automorphe et compactifications partielles de certaines vari\'et\'es de Griffiths-Schmid}, Compos. Math. {\bf 141} (2005), 1081--1102.
\bibitem[CK]{CK} E. Cattani and A. Kaplan, {\it Polarized mixed Hodge structures and the local monodromy of a variation of Hodge structure}, Invent. Math. {\bf 67} (1982), no. 1, 101--115.
\bibitem[CCK]{CCK} J. A. Carlson, E. Cattani and A. Kaplan, {\it Mixed Hodge structures and compactification of Siegel's space}, {\rm in Journ\'ees de g\'eom\'etrie alg\'ebrique d'Angers 1979 (A. Beauville, ed.), Sijthohoff \& Nordhoff}, 1980, 77--105.
\bibitem[CKS]{CKS}E. Cattani, A. Kaplan and W. Schmid, {\it Degeneration of Hodge structures}, Ann. of Math. {\bf 123} (1986), 457--535.
\bibitem[G1]{G} P. Griffiths, {\it Periods of integrals on algebraic manifolds. I. Construction and properties of the modular varieties}, {\rm Amer. J. Math. {\bf 90} (1968) 568--626.}
\bibitem[G2]{G2} P. Griffiths, {\it Hodge theory and representation theory} (Lectures at Texas Christian University, June 18--22, 2012), available at \url{ http://faculty.tcu.edu/gfriedman/CBMS2012/}
\bibitem[GGK]{GGK}M. Green, P. Griffiths and M. Kerr, {\it  Mumford-Tate groups and domains: their geometry and arithmetic}, {\rm Annals of Math Studies, {\bf 183}. Princeton University Press, 2012.} 
\bibitem[GRT]{GRT}P. Griffiths, C. Robles and D. Toledo, {\it Quotients of non-classical flag domains are not algebraic}, arXiv:1303.0252.
\bibitem[H1]{H1} T. Hayama, {\it On the boundary of the moduli spaces of log Hodge structures: triviality of the torsor}, Nagoya Math. J. {\bf 198} (2010), 173-190.
\bibitem[H2]{H} T. Hayama, {\it Boundaries of cycle spaces and degenerating Hodge structures}, to appear in Asian J. Math., 
{arXiv:1203.6770}
\bibitem[KU]{KU} K. Kato and S. Usui, {\it Classifying space of degenerating polarized Hodge structures}, {\rm  Annals of Mathematics Studies, {\bf 169}. Princeton University Press, Princeton, NJ, 2009.}
\bibitem[KP]{KP} M. Kerr and G. Pearlstein, {\it  Boundary components of Mumford-Tate domains}, {\rm arXiv:1210.5301}.
\bibitem[N]{N} Y. Namikawa, {\it Toroidal compactification of Siegel spaces}, {\rm Lecture Notes in Math} {\bf 812}, {\rm Springer-Verlag}, 1980.
\bibitem[S]{S} W. Schmid, {\it Variation of Hodge structure: the singularities of the period mapping}, {\rm Invent.\ Math.\ {\bf 22} (1973), 211--319.}
\end{thebibliography}
\end{document}